\documentclass[12pt]{amsart}

\usepackage{amsmath,amssymb,mathrsfs}

\usepackage{amsthm}

\usepackage{amsthm} 
\usepackage{amsfonts}
\usepackage{eucal}

\numberwithin{equation}{section}

\newtheorem{thm}{Theorem}[section]
\newtheorem{lemma}[thm]{Lemma}

\newtheorem*{proof-mthm}{Proof of Main Theorem}

\newcommand{\C}{{\mathbb C}}

 \newcommand{\blem}{\begin{lemma}}
 \newcommand{\elem}{\end{lemma}}
\newcommand{\bpve}{\begin{proof}}
\newcommand{\epve}{\end{proof}}

\def\C{\mathbb{C}}

\def\Re{{\text{\rm \,Re\,}}}

\usepackage{amssymb}
\usepackage{amsmath}
\usepackage{amsfonts}
\usepackage[colorlinks=true,linkcolor=red]{hyperref}
\newtheorem{theorem}{Theorem}
\theoremstyle{plain}

\newtheorem{proposition}{Proposition}

\numberwithin{equation}{section}

\numberwithin{theorem}{section}
\numberwithin{proposition}{section}

\begin{document}

\title{Berezin Transform Commutators}

\author[Borichev]{Alexander Borichev}

\address{Aix-Marseille University, CNRS, I2M, Marseille,
France}

\email{alexander.borichev@math.cnrs.fr}

\author[Fantolini]{G\'erard Fantolini}

\address{Aix-Marseille University, CNRS, I2M, Marseille,
France}

\email{fantolinigerard@gmail.com}

\author[Youssfi]{El-Hassan Youssfi}

\address{Aix-Marseille University, CNRS, I2M, Marseille,
France}

\email{el-hassan.youssfi@univ-amu.fr}


\begin{abstract}
We study commutators of Berezin transforms $\mathcal B_\alpha$ associated with radial weighted Fock spaces 
with exponential weights $e^{-\alpha |z|^{m}}$ on the complex plane. 
We prove that when $m>2$, the commutator $[\mathcal B_\alpha,\mathcal B_\beta]$ is compact  on the space 
$L^{2}(\mathbb C, dA)$ whenever $\alpha,\beta>0$. 
We further show that if $m>2$ and $\alpha\neq\beta$, then the commutator is never of finite rank.  
\end{abstract}

\maketitle

\section{Introduction and statement of the main results}
The spectral theory of operators associated with reproducing kernel
Hilbert spaces plays a central role in modern analysis. In the setting
of Fock type spaces, the Berezin transform provides a natural bridge
between function-theoretic properties of kernels and operator-theoretic
properties such as boundedness and compactness.
Such operators play a fundamental role in 
quantization theory and operator theory on reproducing kernel Hilbert spaces 
\cite{B,E1,E2}. 
On Fock spaces and related weighted spaces, they are closely connected to 
Hankel operators and kernel asymptotics 
\cite{BY,CF,SY,Z}. 

Classical work \cite{Z} of Zhu on Fock spaces established a systematic study of 
the Berezin transform and its boundedness and compactness properties.
Recently, more refined kernel estimates and geometric techniques have
been developed in the study of weighted Fock spaces and reproducing
kernels, notably in the author's work \cite{BFY}, see also the paper \cite{Tian} dealing with the several complex variables case.

Motivated by these developments, we investigate the spectral behaviour
of commutators of Berezin transforms in weighted radial Fock type
spaces. 

In this paper we consider radial weights of the form 
$e^{-\alpha |z|^{m}}$ on $\mathbb C$. 
We prove that when $m\geq 2$, the commutator $[\mathcal B_\alpha,\mathcal B_\beta]$ 
is always compact, and we establish a complementary result showing that 
for $\alpha\neq\beta$ the commutator is never of finite rank unless when $m=2$, extending the result of \cite{BFY}. 

Let $m\ge 2$. Given $\alpha>0$, we consider the generalized Fock space $\mathcal F^m_\alpha$,
$$
\mathcal F^m_\alpha=\bigl\{ F\in\mathrm{Hol}(\mathbb C):\|F\|^2_{\mathcal F^m_\alpha}=\varkappa_{m,\alpha}\int_{\mathbb C}|F(z)|^2e^{-\alpha |z|^{m}}\,dA(z)<\infty\bigr\}, 
$$
where $dA$ is the planar Lebesgue measure, and $\varkappa_{m,\alpha}=\frac{m\alpha^{2/m}}{2\pi\Gamma(2/m)}$ is the normalizing factor. The corresponding Berezin transform $\mathcal B_\alpha=\mathcal B_{m,\alpha}$ is defined via the
reproducing kernel \(K_\alpha(z,w)\) of the space $\mathcal F^m_\alpha$, by the formula 
\[
\mathcal B_\alpha f(z)=\frac{\varkappa_{m,\alpha}}{K_\alpha(z,z)}
\int_{\mathbb C} f(w) |K_\alpha(z,w)|^2
e^{- \alpha |w|^m}\,dA(w).
\]

Throughout the paper we study the operator properties of the commutator
$$
[\mathcal B_\alpha,\mathcal B_\beta]= \mathcal B_\alpha \mathcal B_\beta- \mathcal B_\beta \mathcal B_\alpha.
$$
We start by showing that this operator acts boundedly on the Hilbert space
$L^2(\mathbb C, dA)$. 

Our first main result is as follows:

\begin{theorem}\label{thm:main} Let  $m\geq 2$ and $ \alpha, \beta >0$.
Then the commutator $[\mathcal B_\alpha,\mathcal B_\beta]$ is compact on $L^{2}(\mathbb C, dA)$. 
\end{theorem}

We should point out that the Berezin transform  $\mathcal B_\alpha$  is  bounded  but never compact on  $L^{2}(\mathbb C,  dA(z))$. This fact will be proved herein. 

Our second main result deals with the finite rank property of the commutator.

\begin{theorem}\label{thm:main2}
 Suppose that $m>0$.  Then the following are equivalent:
\begin{itemize}
\item[(1)]
 $[\mathcal B_\alpha,\mathcal B_\beta]$ is of finite rank for all $\alpha, \beta >0$.
 \item[(2)]
 $[\mathcal B_\alpha,\mathcal B_\beta]$ is of finite rank for some $\alpha, \beta >0$ such that $\alpha\not = \beta$.
 \item[(3)]  $[\mathcal B_\alpha,\mathcal B_\beta]=0$ for all $\alpha, \beta >0$.
 \item[(4)]  $[\mathcal B_\alpha,\mathcal B_\beta]=0$ for some $\alpha, \beta >0$ such that $\alpha\not = \beta$. 
\item[(5)] $m=2$. 
 \end{itemize}
\end{theorem}

Since for $m=2$, the commutator vanishes, from now on we deal only with the case $m\not=2$. Then, by the results in \cite{Z,BFY}, $[\mathcal B_\alpha,\mathcal B_\beta]=0$ if and only if 
$\alpha=\beta$. 

\section{Basic facts about the Bergman geometry on weighted Fock spaces}

In this section we recall several geometric and analytic properties of
generalized Fock spaces that will be used throughout the paper. 

Given $m>2$, $\alpha>0$, we start with the weight $\phi_\alpha(z)=\alpha |z|^m$ smoothed in a small neighborhood of the origin. 
Recall that the Berg\-man metric induced by this weight
is given by
\[
ds_\alpha^2(z)=\partial\bar\partial \phi_\alpha(z)\,|dz|^2 .
\]
Consequently the length element of the metric is
\[
ds_\alpha(z)=\sqrt{\partial\bar\partial\phi_\alpha(z)}\,|dz|
= \frac{|dz|}{\tau_\alpha(z)},
\]
where the fundamental quantity associated with the weight is the intrinsic scale
\[
\tau_\alpha(z)=\frac{2}{m\sqrt{\alpha}}|z|^{1-\frac m2} = (\partial\bar\partial\phi_\alpha(z))^{-1/2},\qquad |z|\ge 1.
\]
We have 
\begin{gather*}
\sup_{z\in\mathbb C}\tau_\alpha(z)<\infty,\notag\\
|z|\asymp |w| \implies \tau_\alpha(z)\asymp \tau_\alpha(w),\qquad z,w\in\mathbb C.
\end{gather*}

Furthermore, the intrinsic distance between two points $z,w\in\C$ is defined as
\[
d_\alpha(z,w)
=
\inf_{\gamma}
\int_0^1
\sqrt{\partial\bar\partial\phi_\alpha(\gamma(t))}\,
|\gamma'(t)|\,dt,
\]
where the infimum is taken over all piecewise $C^1$ curves $\gamma$
joining $z$ to $w$.

Equivalently,
\[
d_\alpha(z,w)
=
\inf_{\gamma}
\int_0^1
\frac{|\gamma'(t)|}{\tau_\alpha(\gamma(t))}\,dt .
\]
Furthermore, if $|z|,|w|\ge 1$, then 
\begin{equation}
d_\alpha(z,w)\asymp
\begin{cases}
|z-w|(1+|z|)^{\frac{m}2-1},\qquad  \bigl| |z|-|w| \bigr|<|z|/2 ,\\
|z|^{m/2}+|w|^{m/2},\qquad \bigl| |z|-|w| \bigr|\ge |z|/2.
\end{cases}
\label{tfdist}
\end{equation}

Let  $D_\alpha(z,r)=\{w\in\mathbb C:d_\alpha(z,w)<r\}$ denote the Bergman ball of center \(z\)
and radius \(r>0\). Then 
the Euclidean area of the Bergman ball satisfies
\begin{equation}
|D_\alpha(z,r)|\asymp \min\bigl( |z|^{2-m}r^2, r^{4/m}\bigr),\qquad |z|\ge 1,\,r\ge 1.
\label{tfarea}
\end{equation}

\section{The reproducing kernel and the asymptotic of the Mittag-Leffler function}

Since 
$$
\|z^n\|^2_{\mathcal F^m_\alpha}=\frac{m\alpha^{2/m}}{2\pi\Gamma(2/m)}\int_{\mathbb C}|z|^{2n}e^{-\alpha |z|^m}\,dA(z)= 
\alpha^{-2n/m}
\frac{\Gamma(\frac{2n+2}{m})}{\Gamma(2/m)},
$$
the reproducing kernel $K_\alpha(z,w)$ of the generalized Fock space $\mathcal F^m_\alpha$ is given by 
\[
K_{\alpha}(z,w)
=
\Gamma(2/m)
\sum_{n=0}^\infty
\frac{
\alpha^{2n/m}(z\overline w)^n
}{
\Gamma\left(\frac{2n+2}{m}\right)
}.
\]

Let us consider the two-parameter Mittag--Leffler function
\[
E_{a,b}(u)
=
\sum_{n=0}^\infty
\frac{u^n}{\Gamma(a n+b)}.
\]
We have 
\[
K_{\alpha}(z,w)
=
\Gamma(2/m)
E_{\frac{2}{m},\,\frac{2}{m}}
\left(\alpha^{2/m}z\overline w\right).
\]

In particular, on the diagonal,
\[
K_{\alpha}(z,z)
=
\Gamma(2/m)
E_{\frac{2}{m},\,\frac{2}{m}}
\left(\alpha^{2/m}|z|^2\right).
\]

We use here some standard asymptotical estimates for the Mittag--Leffler function 
\[
  E_{a,a}(u)
  = \sum_{n=0}^\infty \frac{u^n}{\Gamma(an+a)},
\]
with $0<a<1$, which is an entire function of order $1/a > 1$. 

Namely (see \cite[Chapter 1]{Pod}), for small $\varepsilon>0$ we have 
\begin{equation*}
  E_{a,a}(z)
  = \frac{1}{a}\,z^{(1-a)/a}\,e^{z^{1/a}}+o(1),\qquad |z|\to\infty,\,|\arg z|<\varepsilon,
\end{equation*}
and 
$$
  |E_{a,a}(u)|
  \leq C(a)\left(
    |u|^{(1-a)/a}\,e^{\operatorname{Re}(u^{1/a})}
    + \frac{1}{1 + |u|^2}
  \right),\qquad u\in\mathbb C.
$$

These formulas yield 
\begin{equation}
K_{\alpha}(z,w)
\lesssim 
|zw|^{(m/2)-1}e^{\alpha \Re (z\bar w)^{m/2}}+\frac{1}{1+|zw|^2},
\quad z,w\in\mathbb C,\label{k1}\\
\end{equation}
and
\begin{equation}
K_{\alpha}(z,z)=\frac{m}{2}\Gamma(2/m)\alpha^{1-(2/m)}|z|^{m-2}e^{\alpha |z|^m}+o(1),\qquad |z|\to\infty\label{k2}.
\end{equation}

Given $R<\infty$, suppose that $|z-w|\le R|z|^{1-(m/2)}$. Then 
\begin{equation}
K_{\alpha}(z,w)=C(m,\alpha) (z\bar w)^{(m/2)-1}e^{\alpha (z\bar w)^{m/2}}+o(1).\label{ktt}
\end{equation}
Since
$$
2\Re ((z\bar w)^{m/2})-|z|^m-|w|^m=-\bigl| z^{m/2}-w^{m/2}\bigr|^2=O_R(1), 
$$
we obtain that
\begin{equation}
K_{\alpha}(z,w)\asymp_R (1+|z|)^{m-2}e^{(\alpha/2)(|z|^m+|w|^m)},\quad |z-w|\le R|z|^{1-\frac{m}2}.\label{k3}
\end{equation}

These estimates reflect the fact that the kernel $K_\alpha(z,w)$  is strongly localized near
the diagonal and 
satisfies the Gaussian off-diagonal decay with respect to the intrinsic metric.

Passing to the intrinsic metric terms $\tau_\alpha$ and $d_\alpha$ we rewrite our reproducing kernel estimates as follows. 

\begin{lemma}\label{kerlem} Given $m> 2$ and $\alpha>0$ there exists $c>0$ such that 
\begin{align}
|K_\alpha(z,w)| &
\lesssim 
\frac{e^{(\phi_\alpha(z)+\phi_\alpha(w))/2}}
{\tau_\alpha(z)\tau_\alpha(w)}
e^{-c\, d_\alpha(z,w)^2}, \qquad z,w\in\mathbb C,\label{ineq:ker-est}\\
|K_\alpha(z,w)|&
\asymp_R  
\frac{e^{(\phi_\alpha(z)+\phi_\alpha(w))/2}}
{\tau_\alpha(z)^2}, \qquad z,w\in\mathbb C,\, d_\alpha(z,w)<R.\label{tfloc}\\
K_\alpha(z,z)&\asymp \tau_\alpha(z)^{-2}e^{\phi_\alpha(z)}.\label{inequ:diag-Ker-est}
\end{align}
\end{lemma}

\begin{proof} We start with deducing \eqref{ineq:ker-est} from \eqref{k1} and consider several cases depending on $z$ and $w$.
Denote 
$$
Q=\frac{e^{(\phi_\alpha(z)+\phi_\alpha(w))/2}}
{\tau_\alpha(z)\tau_\alpha(w)}
e^{-c\, d_\alpha(z,w)^2}.
$$

If $|z|\le 2$ and $|w|\le 2$, then 
$$
|K_\alpha(z,w)|\lesssim 1\lesssim Q.
$$

If $|z|\le 1$, $|w|>2$, then $d^2_\alpha(z,w)\asymp |w|^m$, and we have 
$$
|K_\alpha(z,w)|\lesssim |w|^{(m/2)-1}e^{\alpha |w|^{m/2}}+1\lesssim  |w|^{(m/2)-1}e^{(\alpha/2)|w|^m}
e^{-\gamma |w|^m}\lesssim Q
$$
for some $\gamma>0$ and some $c>0$.

A similar estimate holds for $|z|>2$, $|w|\le 1$.

From now on we assume that $|z|>1$, $|w|> 1$. If $|z-w|\le |z|/2$, then $d^2_\alpha(z,w)\asymp |z-w|^2\cdot|z|^{m-2}$, and 
\begin{multline*}
|K_\alpha(z,w)|\lesssim |z|^{m-2}e^{\alpha \Re (z\bar w)^{m/2}}+1\\ \lesssim  |z|^{m-2}e^{(\alpha/2)(|z|^m+|w|^m)}
e^{-\gamma |z-w|^2\cdot|w|^{m-2}}\lesssim Q
\end{multline*}
for some $\gamma>0$ and some $c>0$, because 
$$
-2\Re ((z\bar w)^{m/2})+|z|^m+|w|^m=\bigl| z^{m/2}-w^{m/2}\bigr|^2\asymp |z-w|^2\cdot|z|^{m-2}.
$$

From now on we assume that $|z|>1$, $|w|> 1$, $|z-w|>|z|/2$. Then $d^2_\alpha(z,w)\asymp |z|^m+|w|^m$. 

If $\bigl| |z|-|w|\bigr|>|z|/8$, then 
$$
\Re (z\bar w)^{m/2}\le |zw|^{m/2}\le (1/2)(1-\gamma)(|z|^m+|w|^m)
$$
for some $\gamma>0$, because $m>2$.
Therefore, 
\begin{multline*}
|K_\alpha(z,w)|\lesssim |zw|^{(m/2)-1}e^{\alpha \Re (z\bar w)^{m/2}} \\ \lesssim  |zw|^{(m/2)-1}e^{(\alpha/2)(1-\gamma)(|z|^m+|w|^m)}\lesssim Q
\end{multline*}
for some $c>0$.

Finally, if $|z|>1$, $|w|> 1$, $|\arg(z\bar w)|>\frac\pi{20}$, then 
$$
\Re (z\bar w)^{m/2}\le (1-\gamma)|zw|^{m/2}\le (1/2)(1-\gamma)(|z|^{m}+|w|^m)
$$
for some $\gamma>0$. Therefore, 
\begin{multline*}
|K_\alpha(z,w)|\lesssim |zw|^{(m/2)-1}e^{\alpha \Re (z\bar w)^{m/2}} \\ \lesssim  |zw|^{(m/2)-1}e^{(\alpha/2)(1-\gamma)(|z|^m+|w|^m)}\lesssim Q
\end{multline*}
for some $c>0$.

It remains to say that \eqref{inequ:diag-Ker-est} follows from \eqref{k2}, and \eqref{tfloc} follows from \eqref{k3}.
\end{proof}

Given  $m> 2$, $\alpha>0$, we set 
$$
B_\alpha(z,w)=\varkappa_{m,\alpha}\frac{|K_\alpha(z,w)|^2}{K_\alpha(z,z)}e^{-\alpha|w|^m},\qquad z,w\in\mathbb C.
$$
 
\begin{lemma}\label{kercor} For some $\varepsilon>0$ and $c=c(m,\alpha)>0$ we have 
\begin{multline}
B_\alpha(z,w)= C(m,\alpha) |w|^{m-2}e^{-\alpha|z^{m/2}-w^{m/2}|^2}+O(e^{-c(|z^m|+|w|^m)}),\\ |\arg(z\bar w)|<\varepsilon, \,|z|,|w|\to\infty,\label{kkk} 
\end{multline}
and 
\begin{equation}
B_\alpha(z,w)=O\Bigl( \frac{e^{-c\, d_\alpha(z,w)^2}}{\tau_\alpha(z)\tau_\alpha(w}\Bigr),\qquad z,w\in\mathbb C.\label{kk1}
\end{equation}
Furthermore, given $R<\infty$, we have  
\begin{equation}
B_\alpha(z,w)\asymp \frac1{\tau_\alpha(z)^2},\qquad z,w\in\mathbb C,\, d_\alpha(z,w)<R.\label{kk2}
\end{equation}
\end{lemma}

\begin{proof} 
Equality \eqref{kkk} follows from the definition of $B_\alpha$ and estimate \eqref{k2}. 

By Lemma~\ref{kerlem}, for some $c>0$ we have 
$$
B_\alpha(z,w)=O\Bigl( \frac{e^{-c\, d_\alpha(z,w)^2}}{\tau_\alpha(w)^2}\Bigr),\qquad z,w\in\mathbb C.
$$
Now, if $|z-w|<|z|/2$, then $\tau_\alpha(z)\asymp \tau_\alpha(w)$, and \eqref{kk1} follows.
Otherwise, $d_\alpha(z,w)^2\asymp |z|^m+|w|^m$, 
$$
\frac{e^{-c\, d_\alpha(z,w)^2}}{\tau_\alpha(w)^2}=O\Bigl( \frac{e^{-(c/2)\, d_\alpha(z,w)^2}}{\tau_\alpha(z)\tau_\alpha(w)}\Bigr),
$$
and \eqref{kk1} follows again.

Finally, \eqref{kk2} follows from \eqref{kkk}. 
\end{proof} 
 

Later on, we need another result concerning the behavior of the combination of reproducing kernels for different parameters $\alpha$. 
 
Since the Bergman metrics associated with $\phi_\alpha$ and $\phi_\beta$
are comparable, the corresponding intrinsic distances satisfy
\[
d_\alpha(z,w)\asymp d_\beta(z,w),
\]
and similarly
\[
\tau_\alpha(z)\asymp \tau_\beta(z).
\]

Given  $m> 2$, $\alpha,\beta>0$, we set 
$$
G_{\alpha,\beta}(z,u,w)=B_\alpha(z,u)B_\beta(u,w)-B_\beta(z,u)B_\alpha(u,w),\qquad z,u,w\in\mathbb C.
$$

\begin{lemma}\label{stronglem} Let $m>2$, $\alpha,\beta>0$. Then for some $c>0$ we have 
\begin{equation}
|G_{\alpha,\beta}(z,u,w)|\lesssim \frac{e^{-c( d_\alpha(z,u)^2+d_\alpha(u,w)^2)}}{\tau_\alpha(z)\tau_\alpha(w)\tau_\alpha(u)^2}.
\label{l32}
\end{equation}
Furthermore, there exists $C=C(m,\alpha,\beta)$ such that for every $R<\infty$ we have 
\begin{multline}
G_{\alpha,\beta}(z,u,w)=C |uw|^{m-2}\times \\ \times 
\Bigl(e^{-\alpha|\tilde z-\tilde u|^2-\beta|\tilde u-\tilde w|^2}-e^{-\beta|\tilde z-\tilde u|^2-\alpha|\tilde u-\tilde w|^2}\Bigr)\\
+O(e^{-c|z^m|}),\qquad u,w\in D_\alpha(z,R),\label{t5}
\end{multline}
where $\tilde z=z^{m/2}$, $\tilde u=u^{m/2}$, $\tilde w=w^{m/2}$.
\end{lemma}

\begin{proof} \eqref{l32} follows from \eqref{kk1}, \eqref{t5} follows from \eqref{kkk}.
%
\end{proof} 




\section{Boundedness of the Berezin transform}

Given $m> 2$, $\alpha>0$, let $ d\nu_\alpha(z) := e^{-\alpha |z|^m}dA(z)$.   

The Berezin transform $\mathcal B_\alpha$ is defined on $L^{2}(\mathbb C,dA)$ by
\begin{multline*}
(\mathcal B_\alpha f)(z)=\varkappa_{m,\alpha}\int_{\C} \frac{|K_\alpha(z,w)|^2}{K_\alpha(z,z)}f(w)\,d\nu_\alpha(w)\\=\int_{\C} B_\alpha(z,w)f(w)\,dA(w),
\end{multline*}
where
\[
B_\alpha(z,w)=\varkappa_{m,\alpha}\frac{|K_\alpha(z,w)|^2}{K_\alpha(z,z)}e^{-\alpha|w|^m},\qquad z,w\in\mathbb C.
\]


This Gaussian decay \eqref{kk1} of the kernel $B_\alpha(z,w)$ with respect to the intrinsic metric plays a
central role in the operator estimates used in the sequel. 
In this section, we first prove the following:

\begin{proposition}\label{bdd:Ber} 
The Berezin transform $\mathcal B_\alpha$ is bounded on $L^2(\mathbb C, dA)$.   
\end{proposition}

\begin{proof} We use the Schur test in the form 
$$
\|\mathcal B_\alpha\|^2\le \sup_{z\in\mathbb C}\int_{\mathbb C}B_\alpha(z,w)\,dA(w)\cdot \sup_{w\in\mathbb C}\int_{\mathbb C}B_\alpha(z,w)\,dA(z). 
$$
Thus, we need only to verify that 
$$
\sup_{z\in\mathbb C}\int_{\mathbb C}B_\alpha(z,w)\,dA(w)+\sup_{w\in\mathbb C}\int_{\mathbb C}B_\alpha(z,w)\,dA(z)<\infty.
$$

By \eqref{kk1} it remains to establish the following result. 

\begin{lemma}\label{lemij} Let $m>2$, $\alpha,c>0$. Then
\begin{equation*}
\sup_{z\in\mathbb C}\int_{\mathbb C}\frac{\exp(-c\, d_\alpha(z,w)^2)}
{\tau_\alpha(z)\tau_\alpha(w)}\,dA(w)<\infty.
\end{equation*}
\end{lemma}

\begin{proof} Without loss of generality, $|z|\ge 1$. Given $z\in\mathbb C$, we decompose the plane into intrinsic annuli centered at $z$:
\begin{align*}
E_0(z)&=\{w:d_\alpha(z,w)<1\},\\
E_j(z)&=\{w:j\le d_\alpha(z,w)<j+1\},\qquad j\ge 1.
\end{align*}

Then
\begin{multline*}
I(z)
=\frac{1}{\tau_\alpha(z)}
\sum_{j\ge 0}
\int_{E_j(z)}
\frac{1}{\tau_\alpha(w)}
e^{-c\,d_\alpha(z,w)^2}\,dA(w)\\ \le 
\frac{1}{\tau_\alpha(z)}\sum_{j\ge 0}e^{-cj^2}\cdot\Bigl(\sup_{d_\alpha(z,w)<j+1}\frac{1}{\tau_\alpha(w)}\Bigr)\cdot 
\int_{E_j(z)}\,dA(w).
\end{multline*}

If $j\le \delta |z|^{m/2}$ for some small fixed $\delta>0$, then by \eqref{tfdist}, $|w|\asymp|z|$, and hence, 
$\tau_\alpha(w)\asymp\tau_\alpha(z)$, $w\in E_j(z)$. Next, by \eqref{tfarea}, the area of $E_j(z)$ is bounded from above by $C(1+j)^2|z|^{2-m}$. 

On the opposite side, if $j> \delta |z|^{m/2}$, then we use that $\tau_\alpha(w)\gtrsim 1$ and that, by \eqref{tfarea}, the area of $E_j(z)$ is bounded from above by $j^{4/m}$.

Summing up, we obtain that 
\begin{multline*}
I(z)\lesssim  |z|^{(m/2)-1}
\Bigl(\sum_{0\le j\le \delta|z|^{m/2}}e^{-cj^2}|z|^{(m/2)-1}\cdot (1+j)^2|z|^{2-m}\\+
\sum_{j> \delta|z|^{m/2}}e^{-cj^2}j^{4/m}\Bigr)<C,
\end{multline*}
with $C$ independent of $z$.

This completes the proof of Lemma.
\end{proof} 

Since the Schur test is satisfied, this completes the proof of Proposition.
\end{proof}

Our second result in this section concerns the compactness of the Berezin transform.

\begin{proposition}
The operator $\mathcal B_\alpha$ is not
compact on the Hilbert space
$L^2(\mathbb C,dA)$.  
\end{proposition}

\begin{proof} As above, we use that 
\[
(\mathcal B_\alpha f)(z)=\int_{\C} B_\alpha(z,w)f(w)\,dA(w). 
\]
By \eqref{kk2}, there exists $c_0>0$ such that 
\begin{equation}
B_\alpha(z,w)
\ge
\frac{c_0}
{\tau_\alpha(w)^2},\qquad z,w\in\mathbb C,\, d_\alpha(z,w)<1.
\label{tfu1}
\end{equation}

Choose a sequence of points $a_n\in\mathbb C$, $n\ge 1$, with $|a_n|\to\infty$ such that
the intrinsic balls
\[
D_n=D_\alpha(a_n,1/2) 
\]
are pairwise disjoint, and define
\[
f_n=|a_n|^{(m-2)/2}\mathbf 1_{D_n}.
\]
Then $\|f_n\|_{L^2(\mathbb C,dA)}\asymp 1$, and since the sets $D_n$ escape to
infinity, the sequence $(f_n)$ converges weakly to $0$ in $L^2(\mathbb C,dA)$.

If $z,w\in D_n$, then
\[
d_\alpha(z,w)\le d_\alpha(z,a_n)+d_\alpha(a_n,w)<1,
\]
and by \eqref{tfu1}
\[
B_\alpha(z,w)
\ge
\frac{c_0}{\tau_\alpha(w)^2}.
\]
Since $\tau_\alpha$ varies slowly on intrinsic balls,
\[
\tau_\alpha(w)\asymp \tau_\alpha(a_n),\qquad w\in D_n,
\]
and we obtain 
\[
B_\alpha(z,w)
\ge
\frac{c_1}{\tau_\alpha(a_n)^2}.
\]

Therefore,
\begin{multline*}
(\mathcal B_\alpha f_n)(z)=|a_n|^{(m-2)/2}
\int_{D_n}
B_\alpha(z,w)\,dA(w)\\
\ge 
\frac{c_1|a_n|^{(m-2)/2}}{ \tau_\alpha(a_n)^2}|D_n|,\qquad z\in D_n.
\end{multline*}
Sine the Euclidean area of $D_n$ satisfies $|D_n|\asymp |a_n|^{2-m}$, it follows that
\[
(\mathcal B_\alpha f_n)(z)\ge c_2|a_n|^{(m-2)/2},\qquad z\in D_n.
\]

Hence, 
\[
\|\mathcal B_\alpha f_n\|_{L^2(\mathbb C,dA)}^2
\ge
\int_{D_n}|\mathcal B_\alpha f_n(z)|^2\,dA(z)
\ge
c_3>0,\qquad n\ge 1.
\]

Thus the sequence $(f_n)$ converges weakly to $0$, but $\|\mathcal B_\alpha f_n\|_{L^2(\mathbb C,dA)}\not\to0$. A compact operator sends
weakly convergent sequences to norm convergent sequences, so $\mathcal B_\alpha$
cannot be compact on $L^2(\mathbb C,dA)$.
\end{proof}

\section{Compactness of the commutator}

We now use the local expansion and the global estimates on the kernels $B_\alpha$, $G_{\alpha,\beta}$ in Lemmata~\ref{kercor} and \ref{stronglem} to prove the compactness of the commutator.

\begin{theorem}\label{thm:compactness}
Let $m>2$. For every
$\alpha,\beta>0$, the commutator
\[
\mathcal C_{\alpha,\beta}=\mathcal B_\alpha \mathcal B_\beta-\mathcal B_\beta \mathcal B_\alpha
\]
is compact on $L^2(\C,dA)$.
\end{theorem}

To prove this theorem we need some preparatory results.
Denote by $C_{\alpha,\beta}(z,w)$ the kernel of the integral operator $\mathcal C_{\alpha,\beta}$ acting on $L^2(\C,dA)$. Then we have
\begin{multline*}
C_{\alpha,\beta}(z,w)=\int_{\C}\Bigl(B_\alpha(z,u)B_\beta(u,w)-B_\beta(z,u)B_\alpha(u,w)\Bigr)\,dA(u)\\=\int_{\C}G_{\alpha,\beta}(z,u,w) \,dA(u).
\end{multline*}

Given $\alpha>0$, $c>0$ consider an auxiliary kernel
$$
L(z,w)=\frac{e^{-c\,d_\alpha(z,w)^2}}{\tau_\alpha(z)\tau_\alpha(w)}.
$$

\begin{lemma}\label{lem:K0}
The kernel $L$ defines a bounded integral operator $\mathcal L$ on $L^2(\C,dA)$ and its truncated version
$$
L^{M}(z,w):=\mathbf 1_{\{d_\alpha(z,w)>M\}}L(z,w) 
$$
defines a bounded integral operator $\mathcal L^{M}$ on $L^2(\C,dA)$ satisfying 
\[
\|\mathcal L^{M}\|\to 0, \qquad M\to\infty.
\]
\end{lemma}

\begin{proof}
By the proof of Proposition \ref{bdd:Ber}, $L$ defines a bounded operator on $L^2(\C,dA)$.
Furthermore, by analogy with the proof of Lemma~\ref{lemij}, we have
\begin{multline*}
\int_{\C} L^{M}(z,w)\,dA(w)
\lesssim  |z|^{(m/2)-1}\times\\ \times
\Bigl(\sum_{M\le j\le \delta|z|^{m/2}}e^{-cj^2}|z|^{(m/2)-1} (1+j)^2|z|^{2-m}+
\sum_{j> \max(M,\delta|z|^{m/2})}e^{-cj^2}j^{4/m}\Bigr)\\ =u(z,M), 
\end{multline*}
with 
$$
\lim_{M\to\infty}\sup_{z\in \mathbb C}u(z,M)=0.
$$
This, together with the Schur test, completes the proof.
\end{proof}

\begin{lemma}[Far-off-diagonal estimate]
\label{lem:far}
Fix $M>0$ and define the kernel 
\[
C^M_{\alpha,\beta}(z,w):=\mathbf 1_{\{d_\alpha(z,w)>M\}}\,C_{\alpha,\beta}(z,w).
\]
Then the corresponding integral operator $\mathcal C^M_{\alpha,\beta}$ satisfies 
\[
\|\mathcal C^M_{\alpha,\beta}\| \to 0,\qquad M\to\infty.
\]
\end{lemma}

\begin{proof}

We start with the pointwise domination of the commutator kernel.

Recall that
\[
C_{\alpha,\beta}(z,w)=\int_{\C}G_{\alpha,\beta}(z,u,w)\,dA(u)
\]
so that by \eqref{l32} we have 
$$
|C_{\alpha,\beta}(z,w)|\lesssim \frac{1}{\tau_\alpha(z)\tau_\alpha(w)}
\int_{\C}
\frac{e^{-c(d_\alpha(z,u)^2+d_\alpha(u,w)^2)}}{\tau_\alpha(u)^2}
\,dA(u).
$$

By the triangle inequality we have 
$$
d_\alpha(z,u)^2+d_\alpha(u,w)^2 \ge \tfrac12 d_\alpha(z,w)^2,
$$
and hence, 
\[
e^{-c(d_\alpha(z,u)^2+d_\alpha(u,w)^2}
\le
e^{-\frac c4 d_\alpha(z,w)^2-\frac c4d_\alpha(z,u)^2}.
\]

Arguing as in the proof of Lemma~\ref{lemij}, we obtain that 
$$
\sup_{z\in\mathbb C}\int_{\C}
\frac{e^{-(c/4)d_\alpha(z,u)^2}}{\tau_\alpha(u)^2}
\,dA(u)<\infty.
$$

Therefore,
\[
|C_{\alpha,\beta}(z,w)|\lesssim 
\frac{e^{-(c/4) d_\alpha(z,w)^2}}{\tau_\alpha(z)\tau_\alpha(w)}=L(z,w),
\]
(with a different $c$). 

Therefore,
\[
|C^M_{\alpha,\beta}(z,w)|\lesssim\mathbf 1_{\{d(z,w)>M\}} L(z,w)=L^{M}(z,w).
\]
so that by Lemma~\ref{lem:K0} we obtain 
\[
\|\mathcal C^M_{\alpha,\beta}\|\to 0,\qquad M\to\infty.
\]
\end{proof}

The next lemma provides the near diagonal cancellation at large distance from the origin.

\begin{lemma}\label{lem:near-diag}
Given $M>0$, we have 
\[
\sup_{w\in D_\alpha(z,M)}|C_{\alpha,\beta}(z,w)|
=o(\tau_\alpha(z)^{-2}),\qquad |z|\to\infty.
\]
\end{lemma}

\begin{proof} Let $w\in D_\alpha(z,M)$ and let $R>5M$. By \eqref{l32} we have 
\begin{multline*}
|C_{\alpha,\beta}(z,w)|\lesssim \Bigl| \int_{|u^{m/2}-z^{m/2}|\le R} G_{\alpha,\beta}(z,u,w)\,dA(u)\Bigr| \\+\int_{\mathbb C\setminus D_\alpha(z,R/2)}\frac{e^{-c\, d_\alpha(z,u)^2}}{\tau_\alpha(z)^2\tau_\alpha(u)^2}\,dA(u)=I_1+I_2.
\end{multline*}

By the proof of Lemma~\ref{lem:K0},
\begin{equation}
I_2\le \frac{\delta(R)}{\tau_\alpha(z)^2},\qquad |z|>C(R),
\label{t6}
\end{equation}
with $\delta(R)\to 0$ as $R\to\infty$.

By \eqref{t5}, we have
\begin{multline*}
I_1=C|w|^{m-2}\Bigl| \int_{|u^{m/2}-z^{m/2}|\le R} |u|^{m-2} \Bigl(e^{-\alpha|\tilde z-\tilde u|^2-\beta|\tilde u-\tilde w|^2}\\-e^{-\beta|\tilde z-\tilde u|^2-\alpha|\tilde u-\tilde w|^2}\Bigr)\,dA(u)\Bigr|+O(e^{-c|z^m|}).
\end{multline*} 
Passing to the variable $\tilde u=u^{m/2}$, we obtain 
\begin{multline*}
I_1=C_1|w|^{m-2}\Bigl| \int_{|\tilde u-\tilde z|\le R}  \Bigl(e^{-\alpha|\tilde z-\tilde u|^2-\beta|\tilde u-\tilde w|^2}-e^{-\beta|\tilde z-\tilde u|^2-\alpha|\tilde u-\tilde w|^2}\Bigr)\,dA(\tilde u)\Bigr|\\+O(e^{-c|z^m|}).
\end{multline*} 
 
Since
$$
 \int_{\mathbb C}  \Bigl(e^{-\alpha|\tilde z-\tilde u|^2-\beta|\tilde u-\tilde w|^2}-e^{-\beta|\tilde z-\tilde u|^2-\alpha|\tilde u-\tilde w|^2}\Bigr)\,dA(\tilde u)=0,
$$
we conclude that 
$$
I_2\le \frac{\delta(R)}{\tau_\alpha(z)^2},\qquad |z|>C(R),
$$
with $\delta(R)\to 0$ as $R\to\infty$.

Together with \eqref{t6} this completes the proof. 
\end{proof}

Set
\[
C^{R,M}_{\alpha,\beta}(z,w)=\mathbf 1_{\{d(z,w)\le M,\, |z|\ge R\}}\,C_{\alpha,\beta}(z,w),
\]
and consider the corresponding integral operator $\mathcal C^{R,M}_{\alpha,\beta}$.

By Lemma~\ref{lem:near-diag}, we have
\[
|C^{R,M}_{\alpha,\beta}(z,w)|= o(L(z,w)), \qquad R\to\infty.
\]
Therefore, for fixed $M$
\begin{equation}
\|\mathcal C^{R,M}_{\alpha,\beta}\|=o(1), \qquad R\to\infty.
\label{tl}
\end{equation}

\begin{lemma}\label{lem:int} The integral operator $\mathcal C^{{\rm int},R,M}_{\alpha,\beta}$ defined by the kernel 
\[
C^{{\rm int},R,M}_{\alpha,\beta}(z,w)=\mathbf 1_{\{d(z,w)\le M,\, |z|<R\}}\,C(z,w)
\]
is a Hilbert--Schmidt operator on $L^2(\C, dA)$. 
\end{lemma}

\begin{proof} Since the kernel $C_{\alpha,\beta}$ is continuous and the support of $C^{{\rm int},R,M}_{\alpha,\beta}$ is compact, 
this latter kernel is square
integrable. Therefore $\mathcal C^{{\rm int},R,M}_{\alpha,\beta}$ is a Hilbert--Schmidt operator.
\end{proof}

\begin{proof}[Proof of Theorem \ref{thm:compactness}]
Let $\varepsilon>0$. By Lemma~\ref{lem:far} we can find $M$ such that $\|\mathcal C^M_{\alpha,\beta}\|\le \varepsilon$. 
Next, by \eqref{tl}, we can find $R$ such that $\|\mathcal C^{R,M}_{\alpha,\beta}\|\le \varepsilon$. Since 
\[
\mathcal C_{\alpha,\beta}=\mathcal C^M_{\alpha,\beta}+\mathcal C^{R,M}_{\alpha,\beta}+\mathcal C^{{\rm int},R,M}_{\alpha,\beta},
\]
and the latter operator is compact, we obtain that $C$ is a norm limit of compact operators,
hence compact.
\end{proof}


\section{Finite rank property of the Berezin commutator}

We now show that, in the  case \(m\not= 2\), the commutator is not
of finite rank whenever \(\alpha\neq\beta\). The argument uses the Taylor series expansion of the Berezin kernel, the fact that the Berezin tranform preserves the angular mode: 
if $f(z)=z^Nf_1(|z|)$, then $\mathcal B_\alpha f=g$ with $g(z)=z^Ng_1(|z|)$ for some $g_1$, and some results from \cite{BFY}.

\begin{thm}\label{thm:finite-rank}
Let $m\not= 2$, $\alpha, \beta>0$.  If  $\alpha\not = \beta$, then the commutator $\mathcal C_{\alpha,\beta}$ is not of finite rank on the space
$L^2(\mathbb C,dA)$. 
\end{thm}

\begin{proof} As in \cite{BFY} we start by defining the Stieltjes moments
$$
s_{n,\alpha}= \alpha^{-2n/m}\Gamma\Bigl(\frac{2n+2}m\Bigr)=\frac{m\alpha^{2/m}}{2\pi}\int_{\mathbb C} |z|^{2n}  
 e^{-\alpha |z|^m}\,dA(z),\quad n\ge 0,
$$
and an entire function
$$
H_\alpha(z)= \sum_{n\ge0}\frac{z^n}{s_{n,\alpha}}. 
$$
We have $H_\alpha(0)=s_{0,\alpha}^{-1}=\Gamma(2/m)^{-1}$.

Now, 
$$
K_{\alpha}(z,w)=\Gamma(2/m)H_\alpha(z\bar w),
$$
and, hence,
$$
K_{\alpha}(z,z)=\Gamma(2/m)H_\alpha(|z|^2).
$$

Consequently, the Berezin transform $\mathcal B_{\alpha}$ can be expressed as 
\begin{multline*}
(\mathcal B_{\alpha}f)(z)=\varkappa_{m,\alpha}\int_{\mathbb C}f(w)\frac{|K_{\alpha}(w,z)|^2}{K_{\alpha}(z,z)}e^{-\alpha|w|^m}\,dA(w)\\
 = \frac{m\alpha^{2/m}}{2\pi H_{\alpha}(|z|^2)}  \int_{\mathbb C}f(w) |H_{\alpha}(z\overline{w})|^2
 e^{-\alpha |w|^m}\,dA(w).
\end{multline*}

Let $N\in \mathbb Z_+$, and let $f(z)=z^Ng(|z|)$. Then 
\begin{multline*}
(\mathcal B_{\alpha}f)(z)\\= \frac{m\alpha^{2/m}}{2\pi H_{\alpha}(|z|^2)} \sum_{j,k\ge 0} \frac1{s_{j,\alpha}s_{k,\alpha}}\int_{\mathbb C}  (z\bar w)^j (\bar zw)^k w^Ng(|w|)
 e^{-\alpha |w|^m}\,dA(w)\\=z^N\frac{m\alpha^{2/m}}{2\pi H_{\alpha}(|z|^2)} \sum_{k\ge 0} \frac{|z|^{2k}}{s_{k,\alpha}s_{k+N,\alpha}}\int_{\mathbb C} |w|^{2k+2N}  g(|w|)
 e^{-\alpha |w|^m}\,dA(w).
\end{multline*}

Next,
$$
\lim_{z\to 0}\frac{(\mathcal B_{\alpha}f)(z)}{z^N}=\frac{m\alpha^{2/m}}{2\pi  s_{N,\alpha}}    \int_{\mathbb C} |w|^{2N}  g(|w|)
 e^{-\alpha |w|^m}\,dA(w).
$$

Now, let $F_{N,\delta}(z)=z^Ne^{-\delta |z|^m}$. Since 
\begin{multline*}
\int_{\mathbb C} |\zeta|^{2k+2N}  
 e^{-(\alpha+\delta) |\zeta|^m}\,dA(\zeta)=\\ \frac{2\pi}{m}(\alpha+\delta)^{-(2k+2N+2)/m}\Gamma((2k+2N+2)/m)\\=\frac{2\pi}{m}
 \frac{\alpha^{(2k+2N)/m}s_{k+N,\alpha}}{(\alpha+\delta)^{(2k+2N+2)/m}},
\end{multline*}
we have
\begin{multline*}
(\mathcal B_{\alpha}F_{N,\delta})(z)\\=z^N\frac{m\alpha^{2/m}}{2\pi H_{\alpha}(|z|^2)} \sum_{k\ge 0} \frac{|z|^{2k}}{s_{k,\alpha}s_{k+N,\alpha} }\int_{\mathbb C} |w|^{2k+2N}   
 e^{-(\alpha+\delta) |w|^m}\,dA(w)\\=\frac{z^N}{H_{\alpha}(|z|^2)} \sum_{k\ge 0} \Bigl(\frac{\alpha}{\alpha+\delta}\Bigr)^{(2k+2N+2)/m}\cdot\frac{|z|^{2k}}{s_{k,\alpha}}.
 \end{multline*}

Therefore, 
\begin{multline*}
\lim_{z\to 0}\frac{(\mathcal B_{\alpha}\mathcal B_{\beta}F_{N,\delta})(z)}{z^N}=\frac{m\alpha^{2/m}}{2\pi  s_{N,\alpha}}   
\int_{\mathbb C} |w|^{2N} \frac{(\mathcal B_{\beta}F_{N,\delta})(w)}{w^N}  e^{-\alpha |w|^m}\,dA(w)\\=
\frac{m\alpha^{2/m}}{2\pi  s_{N,\alpha}}   
\sum_{k\ge 0} \Bigl(\frac{\beta}{\beta+\delta}\Bigr)^{(2k+2N+2)/m} \cdot \frac{1}{s_{k,\beta}}\int_{\mathbb C} \frac{|w|^{2k+2N}}{H_{\beta}(|w|^2)}   e^{-\alpha |w|^m}\,dA(w) \\=
\frac{m(\alpha/\beta)^{(2N+2)/m}}{2\pi \Gamma((2N+2)/m)}   
\sum_{k\ge 0} \Bigl(\frac{\beta^2}{\beta+\delta}\Bigr)^{(2k+2N+2)/m}  
\times \\ \times\frac{1}{\Gamma((2k+2)/m)}\int_{\mathbb C} \frac{|w|^{2k+2N}}{H_{\beta}(|w|^2)}   e^{-\alpha |w|^m}\,dA(w).
\end{multline*}

Now we set 
\begin{align*}
V_{\alpha,\beta}(k)&=\frac{1}{\Gamma(\frac{2k+2}{m})} \int_{\mathbb C}  \frac{  |z|^{2k} e^{-  \beta |z|^m}  }{ H_{\alpha}(|z|^2)}\,dA(z),\\
V_{\beta,\alpha}(k)&=\frac{1}{\Gamma(\frac{2k+2}{m})} \int_{\mathbb C}  \frac{  |z|^{2k} e^{-  \alpha |z|^m}  }{ H_{\beta}(|z|^2)}\,dA(z),
\end{align*}
and 
$$
P_{m,k,N}=\frac{\Gamma((2k+2N+2)/m)}{\Gamma((2N+2)/m)\Gamma((2k+2)/m)}.
$$

If $\mathcal C_{\alpha,\beta}F_{N,\delta}=0$, then 
$$
\lim_{z\to 0}\frac{(\mathcal B_{\alpha}\mathcal B_{\beta}F_{N,\delta})(z)}{z^N}=\lim_{z\to 0}\frac{(\mathcal B_{\beta}\mathcal B_{\alpha}F_{N,\delta})(z)}{z^N},
$$
and hence,
\begin{multline}
\Bigl(\frac{1}{1+\beta/\delta}\Bigr)^{(2N+2)/m}\sum_{k\ge 0}  \Bigl(\frac{\beta^2}{\beta+\delta}\Bigr)^{2k/m} 
P_{m,k,N}V_{\beta,\alpha}(k+N)\\=
\Bigl(\frac{1}{1+\alpha/\delta}\Bigr)^{(2N+2)/m}\sum_{k\ge 0}  \Bigl(\frac{\alpha^2}{\alpha+\delta}\Bigr)^{2k/m} 
P_{m,k,N}V_{\alpha,\beta}(k+N).
\label{kfin}
\end{multline}

Suppose now that $\mathcal C_{\alpha,\beta}$ is of finite rank. 
Then for every $N\ge 0$, and for every $\delta>0$ outside of a finite number of values, we have \eqref{kfin}. 
Passing to the limit $\delta\to\infty$, we obtain 
\begin{equation}
V_{\beta,\alpha}(N)=V_{\alpha,\beta}(N),\qquad N\ge 0. 
\label{eq1}
\end{equation}

Set 
\begin{align*}
U_{\alpha,\beta}(k)&=\alpha^{4k/m}V_{\alpha,\beta}(k),\qquad k\ge 0,\\
U_{\beta,\alpha}(k)&=\beta^{4k/m}V_{\beta,\alpha}(k),\qquad k\ge 0.
\end{align*}

Now, it is time to use some results from \cite{BFY}. 

\begin{lemma}[\cite{BFY}, Lemmata 2.1, 2.2] Let $m,\alpha,\beta>0$. 
\begin{itemize}
\item[(i)] If the commutativity relation 
\begin{equation}
(\mathcal B_\beta \mathcal B_\alpha F_{0,\delta})(0)=(\mathcal B_\alpha \mathcal B_\beta F_{0,\delta})(0)
\label{eq2}
\end{equation}
holds for an infinite sequence of $\delta$, tending to $\infty$, then
$$
U_{\alpha,\beta}(n)=U_{\beta,\alpha}(n),\qquad 0\le n<\frac{m}2.
$$
\item[(ii)]Let $m\not\in 2\mathbb Z_+$, $\alpha\not=\beta$. The commutativity relation \eqref{eq2} cannot hold for any 
infinite sequence of $\delta$, tending to $\infty$.
\end{itemize}
\label{lastlem}
\end{lemma}

By Lemma~\ref{lastlem}~(i), if $m>2$, $\alpha\not=\beta$, and $\mathcal C_{\alpha,\beta}$ is of finite rank, then $U_{\alpha,\beta}(1)=U_{\beta,\alpha}(1)$ which contradicts to 
\eqref{eq1}. By Lemma~\ref{lastlem}~(ii), if $m<2$, $\alpha\not=\beta$, then $\mathcal C_{\alpha,\beta}$ is not of finite rank.

This shows that for $m\not=2$, $\alpha\not=\beta$, the operator $\mathcal C_{\alpha,\beta}$ is not of finite rank.
\end{proof}


\begin{thebibliography}{99}

\bibitem{B}
F.~A.~Berezin,
\textit{Covariant and contravariant symbols of operators},
Math.\ USSR-Izv.\ {\bf 6} (1972), 1117--1151.

\bibitem{BY}
H.~Bommier-Hato, El~H.~Youssfi,
\textit{Hankel operators and the Stieltjes moment problem},
J. Funct.\ Anal.\ {\bf 258} (2010), 978--998.

\bibitem{BFY}
A.~Borichev, G.~Fantolini,  El~H.~Youssfi, \textit{On the commutativity of the Berezin transform}, 
Arch.\ Math.\ (Basel) {\bf 126} (2026), 397--405.

\bibitem{CF}
C.~Castante, J.~F\`abrega, D.~Pascuas,
\textit{Hankel bilinear forms on generalized Fock-Sobolev spaces on $\mathbb C^{n}$},
Ann.\ Acad.\ Sci.\ Fenn.\ Math.\ {\bf 45} (2020), 841--862.

\bibitem{E2}
M.~Engli\v{s},
\textit{Functions invariant under the Berezin transform},
J. Funct.\ Anal.\ {\bf 121} (1994), 233--254.

\bibitem{E1}
M.~Engli\v{s},
\textit{Asymptotics of the Berezin transform and quantization on planar domains},
Duke Math.\ J.\ {\bf 79} (1995), 57--76.

\bibitem{Pod}
I.~Podlubny, \textit{Fractional Differential Equations}, Math. Sci. Engrg., 198
Academic Press, Inc., San Diego, CA, 1999.


\bibitem{SY}
K.~Seip, El.~H.~Youssfi,
\textit{Hankel operators on Fock spaces and related Bergman kernel estimates},
J. Geom.\ Anal.\ {\bf 23} (2013), 170--201.

\bibitem{Tian}
H.~Tian,
\textit{Products of Berezin Transforms on Generalized Fock Spaces}, arXiv 2610.14599.

\bibitem{Z}
K.~Zhu,
\textit{Analysis on Fock spaces},
Springer, 2012.

\end{thebibliography}
\end{document}